\documentclass[final,leqno,onetabnum]{siamltex704}

\usepackage{amsmath,amssymb,exscale,cmmib57} % for \begin{bmatrix}, \begin{smallmatrix}, \begin{align} etc.
\usepackage{amsfonts} % for \mathbb{} etc. `blackboard' for uppercase only.

\usepackage{ifpdf}
\ifpdf %%%%%%%%%%%%%%%%%%%%%%%%%%%%%%%%%%%%%%%%%%%%%%%%%%%%%%%%%% PDFLATEX %%%
\usepackage[pdftex]{graphicx}      %%% graphics for pdfLaTeX. Use epstopdf 
\DeclareGraphicsExtensions{.pdf}   %%% standard extension for included graphics
\usepackage[pdftex]{thumbpdf}      %%% thumbnails for pdflatex. Run thumbpdf at the end
\usepackage[pdftex,                %%% hyper-references for pdflatex
pdfstartview=FitH,%
bookmarks=true,%                   %%% generate bookmarks ...
bookmarksnumbered=true,%           %%% ... with numbers
bookmarksopen=true,%
hypertexnames=false,%              %%% needed for correct links to figures !!!
breaklinks=true,%                  %%% break links if exceeding a single line
  colorlinks=true,%
  linkcolor=blue,anchorcolor=blue,%
  citecolor=blue,filecolor=blue,%
  menucolor=blue]%
{hyperref} 
\hypersetup{
pdfauthor   = {ANDREW V. KNYAZEV AND KLAUS NEYMEYR},
pdftitle    = {Preconditioned eigensolver convergence theory in a nutshell},
pdfsubject  = {Paper},
pdfkeywords = { 
Iterative method, 
preconditioning, 
preconditioner,
eigenvalue, 
eigenvector,
Rayleigh quotient, 
convergence theory, 
gradient iteration.},
}
\else %%%%%%%%%%%%%%%%%%%%%%%%%%%%%%%%%%%%%%%%%%%%%%%%%%%%%%%%%% LATEX %%%
\usepackage[dvips]{graphicx}       %%% graphics for dvips
\DeclareGraphicsExtensions{.eps}   %%% standard extension for included graphics
\usepackage[ps2pdf]{thumbpdf}      %%% thumbnails for ps2pdf
\usepackage[ps2pdf,                %%% hyper-references for ps2pdf
hypertexnames=false,%              %%% needed for correct links to figures !!!
breaklinks=true,%                  %%% break links if exceeding a single line
  colorlinks=true,%
  linkcolor=blue,anchorcolor=blue,%
  citecolor=blue,filecolor=blue,%
  menucolor=blue]%
{hyperref} 
\hypersetup{
pdfauthor   = {ANDREW V. KNYAZEV AND KLAUS NEYMEYR},
pdftitle    = {Preconditioned eigensolver convergence theory in a nutshell},
pdfsubject  = {Paper},
pdfkeywords = { 
Iterative method, 
preconditioning, 
preconditioner,
eigenvalue, 
eigenvector,
Rayleigh quotient, 
convergence theory,
gradient iteration.},
pdfcreator  = {LaTeX with hyperref package},
pdfproducer = {dvips + ps2pdf}
}
\fi 

\renewcommand{\ldots}{\ensuremath{\dotsc}}

\newcommand{\spanl}{\mathrm{span}}
\newcommand{\pgx}{\phi_\gamma(x)}
\newcommand{\cgx}{C_{\pgx}(Bx)}
\newcommand{\pgxa}{\phi_\gamma(|x|)}
\newcommand{\cgxa}{C^R_{\pgxa}(B|x|)}
\newcommand{\px}{\phi_1(x)}
\newcommand{\cx}{C_{\px}(Bx)}
\newcommand{\igk}{I_\gamma(\kappa)}

\title{
Gradient flow approach to geometric 
convergence analysis 
of preconditioned eigensolvers
\thanks{%
Received by the editors June 16, 2008; revised March 5, 2009; 
accepted March 16, 2009; 
published electronically ???????, 2009. 
Preliminary version http://arxiv.org/abs/0801.3099
}
}
\author{Andrew V. Knyazev\thanks{
Department of Mathematical and Statistical Sciences, 
University Colorado Denver, P.O. Box 173364, Campus Box 170, 
 Denver, CO 80217-3364
(andrew.knyazev at ucdenver.edu, http://math.ucdenver.edu/$\sim$aknyazev/).
Supported by the NSF-DMS 0612751.}
\and  
Klaus Neymeyr\thanks{Universit\"at Rostock, Institut f\"ur Mathematik, 
Universit\"atsplatz 1, 18055 Rostock, Germany
(klaus.neymeyr at mathematik.uni-rostock.de,
http://cat.math.uni-rostock.de/$\sim$neymeyr/). 
}}
\begin{document}

\vspace{-1.2in}
%\slugger{simax}{2009}{??}{?}{???--???}

\vspace{.9in}
\setcounter{page}{1}
\maketitle

\begin{abstract}
Preconditioned eigenvalue solvers (eigensolvers) are gaining popularity, 
but their convergence theory remains sparse and complex. 
We consider the simplest preconditioned eigensolver---the gradient iterative method 
with a fixed step size---for symmetric generalized eigenvalue problems, where 
we use the gradient of the  Rayleigh quotient as an optimization direction.  
A sharp convergence rate bound for this method has been obtained 
%by the authors 
in 2001--2003. 
It still remains the only known such bound for any of the methods in this class.
While the bound is short and simple, its proof is not.
We extend the bound to Hermitian matrices in the complex space and 
present a new self-contained %more general 
and significantly shorter proof using novel geometric ideas. 
%We first demonstrate that, for a given initial eigenvector approximation, 
%the next iterative approximation 
%belongs to a cone if we apply all admissible preconditioners. 
%Second, using the classical Temple inequality we easily determine 
%that a subspace spanned by two specific eigenvectors 
%gives the smallest norm of the gradient of the Rayleigh quotient. 
%Then we analyze a corresponding continuation method 
%of a gradient flow of the Rayleigh quotient and show 
%that the slowest convergence of the 
%continuation method is reached when the initial vector  
%belongs to this two-dimensional subspace. 
%Next, we extend this result by integration to our fixed step gradient method 
%to conclude that the point on the cone, which corresponds to the poorest convergence 
%and thus gives the guaranteed convergence rate bound,
%belongs to the same two-dimensional invariant subspace. 
%This reduces the convergence analysis to a two-dimensional case, 
%where a sharp convergence rate bound is derived. 
\end{abstract}



%\begin{DOI} 
(Place for Digital Object Identifier, to get an idea of the final spacing.) 
%?????? \end{DOI}

\pagestyle{myheadings}
\thispagestyle{plain}
\markboth{ANDREW KNYAZEV AND KLAUS NEYMEYR}
{GRADIENT FLOW CONVERGENCE ANALYSIS OF EIGENSOLVERS}
\section{Introduction}
We consider a generalized eigenvalue problem (eigenproblem) for a linear pencil 
$B-\mu A$ with symmetric (Hermitian in the complex case) 
matrices $A$ and $B$ with positive definite $A$. 
%In many applications, e.g.,\ for mechanical vibrations, where 
%$A$ is the stiffness matrix and $B$ is the mass matrix, the matrix 
%$B$ is  also positive definite, and only the largest eigenvalues, 
%characterizing the smallest vibration frequencies, need to be computed.    
%In some applications, e.g., in 
%mechanical buckling of shells with variable-sign curvature, 
%$B$ is not sign-definite and may be singular, both the largest and the smallest eigenvalues
%describing the bucking force under stretching or compression 
%may be of interest. 
The eigenvalues $\mu_i$ are enumerated in decreasing order 
$\mu_1\geq\ldots\geq\mu_{\min}$
and the $x_i$ denote the corresponding eigenvectors.
%We want to compute the largest eigenvalue, $\mu_1$. 
%The problem of computing the smallest eigenvalues  
%can be reformulated as in the previous case by substituting $B$ for $-B$---and 
%our theory is not affected by this substitution. 
The largest value of the Rayleigh quotient $\mu(x)={(x,Bx)}/{(x,Ax)},$ 
where  $(\cdot,\cdot)$ denotes the standard scalar product, is 
the largest eigenvalue $\mu_1$. 
It can be approximated iteratively by maximizing the 
Rayleigh quotient in the direction of its gradient, 
which is proportional to $(B-\mu(x)A)x$.
Preconditioning is used to accelerate the convergence;   
see, e.g.,\ \cite{d,KNY1987,KNY1998,k99,KNN2003} and the references therein.
Here we consider the simplest preconditioned eigenvalue solver 
(eigensolver)---the gradient iterative method with an explicit formula for the step size,
cf. \cite{d}, 
%Denoting an initial eigenvector approximation by $x$, 
%the next iterate is 
one step of which is described by 
\begin{equation}
  \label{e.preeig}
  x'=x+\frac1{\mu(x)-\mu_{\min}} T(Bx-\mu(x) Ax), \quad \mu(x)=\frac{(x,Bx)}{(x,Ax)}.
\end{equation} 
The symmetric (Hermitian in the complex case) positive definite matrix $T$ 
in (\ref{e.preeig}) is called the \emph{preconditioner}. 
Since $A$ and $T$ are  
both positive definite, we assume that
  \begin{equation}
    \label{e.bass}
    (1-\gamma)(z,T^{-1}z)\leq (z,Az) \leq (1+\gamma)(z,T^{-1}z),
                                   \,\forall z, 
\text{ for a given }\gamma\in[0,1).
  \end{equation}
%There are preconditioned eigensolvers, 
%e.g.,\ \cite{k00,k99}, which outperform method \eqref{e.preeig} and 
%do not need knowledge of $\mu_{\min}.$ 
%Our form \eqref{e.bass} of the spectral equivalence condition may not 
%be the most convenient to work with. 
%However, the importance of (\ref{e.preeig}) and (\ref{e.bass}) 
%is that under these assumptions 
The following result is proved in \cite{KNN2003,NEY2001a,NEY2001b} 
for symmetric matrices in the real space.
\begin{theorem}
  \label{t.1}
If $\mu_{i+1}<\mu(x)\leq\mu_i$ then $\mu(x')\geq\mu(x)$ and 
%either $\mu(x')>\mu_i$ or
  \begin{equation}
    \label{e.muest}
    \frac{\mu_i-\mu(x')}{\mu(x')-\mu_{i+1}}
\leq\sigma^2
    \frac{\mu_i-\mu(x)}{\mu(x)-\mu_{i+1}},\quad
\sigma=1 -(1-\gamma)\frac{\mu_i-\mu_{i+1}}{\mu_i-\mu_{\min}}. 
  \end{equation}
The convergence factor $\sigma$ cannot be improved with the chosen terms and assumptions. 
%and assumptions and bound (\ref{e.muest}) is asymptotically sharp as $\mu(x)\to\mu_i$. 
\end{theorem}

Compared to other known non-asymptotic convergence rate bounds 
for similar preconditioned eigensolvers, e.g.,\ 
\cite{bkp96,d,KNY1987,KNY1998}, the advantages of 
\eqref{e.muest} are in its sharpness and elegance. 
Method \eqref{e.preeig} is the easiest preconditioned eigensolver, 
but \eqref{e.muest} still remains the only known sharp bound in these terms 
for any of preconditioned eigensolvers. While bound \eqref{e.muest} 
is short and simple, its proof in \cite{KNN2003} is quite the opposite.
It covers only the real case and is not self-contained---in addition 
it requires most of the material from \cite{NEY2001a,NEY2001b}.
Here we extend the bound to Hermitian matrices 
and give a new much shorter and self-contained proof of Theorem~\ref{t.1},  
which is a great qualitative improvement compared to that of 
\cite{KNN2003,NEY2001a,NEY2001b}. The new proof is not yet as elementary as
we would like it to be; however, it is easy enough to hope that a similar  
approach might be applicable in future work on preconditioned eigensolvers. 

Our new proof is based on novel techniques combined
with some old ideas of \cite{KNY1986,NEY2001a,NEY2001b}. 
We demonstrate that, for a given initial eigenvector approximation $x$, 
the next iterative approximation $x'$ described by \eqref{e.preeig} 
belongs to a cone if we apply any preconditioner satisfying \eqref{e.bass}. 
We analyze a corresponding continuation gradient method 
involving the gradient flow of the Rayleigh quotient 
and show that the smallest gradient norm (evidently leading to the slowest convergence) 
of the continuation method is reached when the initial vector  
belongs to a subspace spanned by two specific eigenvectors, 
namely $x_i$ and $x_{i+1}$. 
This is done by showing that Temple's inequality, which provides a lower bound for the
norm of the gradient $\nabla\mu(x)$, is sharp only in $\spanl\{x_i,x_{i+1}\}$.
Next, we extend by integration the result for the  continuation gradient method 
to our actual fixed step gradient method 
to conclude that the point on the cone, which corresponds to the poorest convergence 
and thus gives the guaranteed convergence rate bound,
belongs to the same two-dimensional invariant subspace $\spanl\{x_i,x_{i+1}\}$.
This reduces the convergence analysis to a two-dimensional case 
for shifted inverse iterations, 
where the sharp convergence rate bound is established.

\section{The proof of Theorem \ref{t.1}}\label{s.geo}
We start with several simplifications: 
\begin{theorem}
\label{t.simp}
We can assume that $\gamma>0$,
$A=I$, $B>0$ is diagonal, eigenvalues are simple, 
$\mu(x)<\mu_i$, and $\mu(x')<\mu_i$ in Theorem \ref{t.1} 
without loss of generality.
\end{theorem}\begin{proof}
First, we observe that method \eqref{e.preeig}
and bound \eqref{e.muest} are evidently both invariant with respect to a real shift 
$s$ if we replace the matrix $B$ with $B+sA$, so without loss of generality we need only 
consider the case  $\mu_{\min}=0$ which makes $B\geq0.$ 
Second, by changing the basis from coordinate vectors to the eigenvectors 
of $A^{-1}B$ we can make $B$ diagonal and $A=I$.
Third, having $\mu(x')\geq\mu(x)$ 
if $\mu(x)=\mu_i$ or $\mu(x')\geq\mu_i$, or both, 
bound \eqref{e.muest} becomes trivial.
The assumption  $\gamma>0$ is a bit more delicate. The vector  $x'$ 
depends continuously on the preconditioner $T$, so we can assume that $\gamma>0$ and extend 
the final bound to the case $\gamma=0$ by continuity. 

Finally, we again use continuity to explain why we can assume that 
all eigenvalues (in fact, we only need $\mu_i$ and $\mu_{i+1}$)
are simple and make $\mu_{\min}>0$ and thus $B>0$ 
without changing anything.  
Let us list all $B$-dependent terms, in addition to all participating eigenvalues,
in method (\ref{e.ep}): $\mu(x)$ and $x'$;  
and in bound (\ref{e.muest}): $\mu(x)$ and $\mu(x')$.
All these terms depend on $B$ continuously if $B$ 
is slightly perturbed into  $B_\epsilon$ with some $\epsilon\to0$, so we  
increase arbitrarily small the diagonal entries of the matrix $B$ 
to make all eigenvalues of $B_\epsilon$ simple and  
$\mu_{\min}>0$. 
If we prove bound \eqref{e.muest} for the matrix $B_\epsilon$ 
with simple positive eigenvalues,
and show that the bound is sharp as $0<\mu_{\min}\to0$ with  $\epsilon\to0$, 
we take the limit $\epsilon\to0$ and by continuity extend the result to the 
limit matrix $B\geq0$ with $\mu_{\min}=0$ and possibly multiple eigenvalues. 
\end{proof}

It is convenient to rewrite \eqref{e.preeig}--\eqref{e.muest} 
equivalently by Theorem \ref{t.simp} as follows%
\footnote{Here and below $\|\cdot\|$ denotes the Euclidean vector norm, i.e.,\ 
$\|x\|^2=(x,x)=x^Hx$ for a real or complex column-vector $x$,
as well as the corresponding induced matrix norm.}  
\begin{align} 
\mu(x)x' = Bx-(I-T)(Bx-\mu(x)x), \quad \mu(x)=\frac{(x,Bx)}{(x,x)},\label{e.ep} 
\\
\|I-T\|\leq\gamma, \quad 0<\gamma<1;\label{e.ass} 
\end{align}
and if $\mu_{i+1}<\mu(x)<\mu_i$ and $\mu(x')<\mu_i$  then $\mu(x')\geq\mu(x)$ and 
  \begin{equation}
    \label{e.muestn}
    \frac{\mu_i-\mu(x')}{\mu(x')-\mu_{i+1}}
\leq\sigma^2
    \frac{\mu_i-\mu(x)}{\mu(x)-\mu_{i+1}},\quad
\sigma=1 -(1-\gamma)\frac{\mu_i-\mu_{i+1}}{\mu_i}=
\gamma+(1-\gamma)\frac{\mu_{i+1}}{\mu_i}. 
  \end{equation}
Now we establish the validity and sharpness
of bound \eqref{e.muestn} assuming \eqref{e.ep} and \eqref{e.ass}. 

\begin{theorem}\label{t.cone}
Let us define%
\footnote{We define angles in $[0,\pi/2]$ between vectors 
by $\cos\angle\{x,y\}=%\cos\angle\{\spanl\{x\},\spanl\{y\}\}=
|(x,y)|/(\|x\|\|y\|)$.} 
$\pgx=\arcsin\left(\gamma{\|Bx-\mu(x)x\|}/{\|Bx\|}\right),$ then  
$\pgx<\pi/2$ and $\angle\{x',Bx\}\leq\pgx$. 
Let $w\neq0$ be defined as the vector constrained by  $\angle\{w,Bx\}\leq\pgx$ and with the 
smallest value $\mu(w)$. Then $\mu(x')\geq\mu(w)>\mu(x)$.
\end{theorem}
\begin{proof}
Orthogonality $(x,Bx-\mu(x)x)=0$ by the Pythagorean theorem implies   
$\|Bx\|^2=\|\mu(x)x\|^2+\|Bx-\mu(x)x\|^2$,
so $\|Bx-\mu(x)x\|<\|Bx\|$, since $\mu(x)>0$ as $B>0$, 
and $\sin\angle\{x,Bx\}=\sin\px={\|Bx-\mu(x)x\|}/{\|Bx\|}<1,$
where $Bx\neq0$ as $B>0$. 
A ball with the radius $\gamma\|Bx-\mu(x)x\|\geq\|I-T\|\|Bx-\mu(x)x\|$ 
by (\ref{e.ass}) centered at $Bx$ contains  $\mu(x)x'$ by \eqref{e.ep},
so $\sin\angle\{x',Bx\}\leq \gamma\|Bx-\mu(x)x\|/\|Bx\|<\gamma<1$. 

The statement  $\mu(x')\geq\mu(w)$ follows directly from the definition of $w$. 
Now, 
\[
0<\frac{(x,Bx)}{\|x\|\|Bx\|}=\cos\px<\cos\angle\{w,Bx\}=\frac{|(w,Bx)|}{\|w\|\|Bx\|}\leq
\frac{(w,Bw)^{1/2}(x,Bx)^{1/2}}{\|w\|\|Bx\|}
\]
as $B>0$, so $\sqrt{\mu(x)}<\sqrt{\mu(w)}$ and $\mu(x)<\mu(w).$
\end{proof}

We denote by $\cgx:=\{y: \angle\{y,Bx\}\leq\pgx\}$ the circular cone
around $Bx$ with the opening angle $\pgx$. 
Theorem \ref{t.cone} replaces  $x'$  with the 
minimizer $w$ of the Rayleigh quotient on the cone $\cgx$ in the rest of the paper, 
except at the end of the proof of Theorem  \ref{t.2D}, where we 
show that bounding below the value $\mu(w)$ instead of 
$\mu(x')$ still gives the sharp estimate. 

Later on, in the proof of Theorem \ref{t.wcurve}, we use an argument that holds 
easily only in the real space, so we need the following last simplification.
\begin{theorem}\label{t.real}
Without loss of generality we can consider only the real case. 
\end{theorem}
\begin{proof}
The key observation is that for our positive diagonal matrix $B$ the Rayleigh quotient 
depends evidently only on the absolute values of the vector components, i.e., 
$\mu(x)=\mu(|x|)$, where the absolute value operation is applied component-wise. 
Moreover, $\|Bx-\mu(x)x\|=\|B|x|-\mu(|x|)|x|\|$ and $\|Bx\|=\|B|x|\|$, so 
$\pgx=\pgxa$. The cone $\cgx$ lives in the complex space, 
but we also need its substitute in the real space.  
Let us introduce the notation  $\cgxa$ for  
the {\em real} circular cone with the  
opening angle $\pgxa$ centered at the {\em real} vector $B|x|.$
Next we show that in the real space we have the inclusion 
$\left|\cgx\right|\subseteq\cgxa$. 

For any complex nonzero vectors $x$ and $y$, we have  
$|(y,Bx)|\leq(|y|,B|x|)$ by the triangle inequality, thus 
$\angle\{|y|,B|x|\}\leq\angle\{y,Bx\}$. 
If $y\in\cgx$ then $\angle\{|y|,B|x|\}\leq\angle\{y,Bx\}\leq\pgx=\pgxa$,
i.e.,\ indeed, $|y|\in\cgxa$, which means that $\left|\cgx\right|\subseteq\cgxa$ 
as required.

Therefore, changing the given vector $x$ to take its absolute value $|x|$ and replacing 
the complex cone $\cgx$ with the real cone $\cgxa$ lead to the relations 
$\min_{y\in\cgx}\mu(y)=\min_{|y|\in\left|\cgx\right|}\mu(|y|)\geq
\min_{|y|\in\cgxa}\mu(|y|)$, 
but does not affect 
the starting Rayleigh quotient $\mu(x)=\mu(|x|).$
This proves the theorem
with the exception of the issue of whether the sharpness
%, except the issue of whether the sharpness 
in the real case implies the sharpness in the complex case;
see the end of the proof of Theorem \ref{t.2D}.
\end{proof}
\begin{theorem}
\label{t.wcurve}
We have $w\in\partial\cgx$ and 
$\exists\,\alpha=\alpha_\gamma(x)>-\mu_i$ such that  
$(B+\alpha I)w=Bx$.
The inclusion $x\in\spanl\{x_i,x_{i+1}\}$ implies $w\in\spanl\{x_i,x_{i+1}\}.$ 
\end{theorem}
\begin{proof}
\begin{figure}%%%%%%%%%%%%%%%%%%%%%%%%%%%%%%%%%%%%%%%%%%%%%%%%%%%%%
\begin{center}
\includegraphics[totalheight=0.2\textheight]{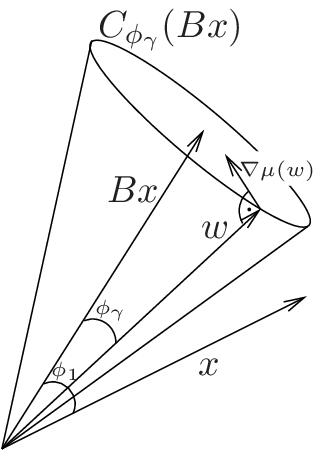}
\end{center}
\caption{\em The cone $\cgx$.}
\label{f1}
\end{figure}%%%%%%%%%%%%%%%%%%%%%%%%%%%%%%%%%%%%%%%%%%%%%%%%%%%%%%%
Assuming that $w$ is strictly inside the cone $\cgx$ 
implies that $w$ is a point of a local minimum of the Rayleigh quotient. 
The Rayleigh quotient has only one local (and global) minimum,  
$\mu_{\min}$, but the possibility 
$\mu(w)=\mu_{\min}$ is eliminated by Theorem \ref{t.cone},  
so we obtain a contradiction, thus $w\in\partial\cgx$. 
 
The necessary condition 
for a local minimum of a smooth real-valued function on a smooth 
surface in a real vector space is that the gradient of the function is orthogonal to 
the surface at the point of the minimum and directed inwards. 
In our case, $\cgx$ is a circular cone  with the axis $Bx$   
and the gradient $\nabla\mu(w)$ is positively proportional 
to $Bw-\mu(w)w$; see Figure \ref{f1}.
We %fix a nonzero vector $x$ and 
first scale the vector $w$ such that $(Bx-w,w)=0$ so that the 
vector $Bx-w$ is an inward normal vector for %the boundary of the cone 
$\partial\cgx$ at the point $w$. This inward normal vector must be 
%, as we have already discussed, 
positively proportional to the gradient, 
$\beta(Bx-w)=Bw-\mu(w)w$ with $\beta>0$, which gives   
$(B+\alpha I)w=\beta Bx$, 
where $\alpha=\beta-\mu(w)>-\mu(w)>-\mu_i$.
Here $\beta\neq0$ as otherwise $w$ would be an eigenvector, but 
$\mu(x)<\mu(w)<\mu(x')$ by Theorem \ref{t.cone}, where by assumptions $\mu_{i+1}<\mu(x)$, 
while $\mu(x')<\mu_i$ by Theorem \ref{t.simp}, which gives a contradiction. 
As the scaling of the minimizer is irrelevant, we denote $w/\beta$ here by $w$ 
with a slight local notation abuse. 

Finally, since $(B+\alpha I)w=Bx$, inclusion 
$x\in\spanl\{x_i,x_{i+1}\}$ gives either the required inclusion 
$w\in\spanl\{x_i,x_{i+1}\}$ or 
$w\in\spanl\{x_i,x_{i+1},x_j\}$ with $\alpha=-\mu_j$ for some 
$j\neq i$ and $j\neq i+1.$ We now show that the latter leads to a contradiction. 
We have just proved that $\alpha>-\mu_i$, thus $j>i+1$. 
Let $x=c_i x_i+c_{i+1}x_{i+1}$, 
where we notice that $c_{i}\neq0$ and $c_{i+1}\neq0$ since 
 $x$ is not an eigenvector.
Then we obtain $w=a_i c_i x_i+a_{i+1}c_{i+1}x_{i+1}+c_j x_j$ where 
$(B-\mu_j)w=Bx$, therefore $a_k=\mu_k/(\mu_k-\mu_j),\,k=i,i+1$. 
Since all eigenvalues are simple, $\mu_{i+1}\neq\mu_j.$ 
We observe that $0<a_i<a_{i+1}$, i.e.,\ 
in the mapping of $x$ to $w$ the coefficient 
in front of $x_i$ changes by a smaller absolute value compared to the change in the 
coefficient in front of $x_{i+1}$. Thus, 
$\mu(x)>\mu(a_i c_i x_i+a_{i+1}c_{i+1}x_{i+1})\geq\mu(w)$
using the monotonicity of the Rayleigh quotient in the absolute 
values of the coefficients of the eigenvector expansion of its argument,
which contradicts $\mu(w)>\mu(x)$ proved in Theorem \ref{t.cone}.
\end{proof}

Theorem \ref{t.wcurve} characterizes the minimizer $w$ of the Rayleigh quotient 
on the cone $\cgx$ for a fixed $x$.
The next goal is to vary $x$, preserving its Rayleigh quotient $\mu(x)$, 
and to determine conditions on $x$ leading to  
the smallest $\mu(w)$ in such a setting. Intuition suggests 
(and we give the exact formulation and 
the proof later in Theorem~\ref{t.2Dred}) that the 
poorest convergence of a gradient method corresponds to the 
smallest norm of the gradient, so in the next theorem 
we analyze the behavior of the gradient $\|\nabla \mu(x)\|$
of the  Rayleigh quotient and the cone opening angle $\pgx.$
\begin{theorem}
\label{t.2Dgrad}
Let $\kappa\in(\mu_{i+1},\mu_i)$ be fixed and   
the level set of the Rayleigh quotient be denoted by 
${\mathcal  L}(\kappa):=\{x\neq0: \mu(x)=\kappa\}$ . 
Both $\|\nabla \mu(x)\|\|x\|$ and $\px-\pgx$ with $0<\gamma<1$
attain their minima on $x\in{\mathcal  L}(\kappa)$ in $\spanl\{x_i,x_{i+1}\}.$
\end{theorem}
\begin{proof} 
By definition of the gradient, 
$\|\nabla\mu(x)\|\|x\|=2\|Bx-\kappa x\|/\|x\|$ 
for $x\in{\mathcal  L}(\kappa)$.
The Temple inequality 
$\|Bx-\kappa x\|^2/\|x\|^2\geq(\mu_i-\kappa)(\kappa-\mu_{i+1})$ 
is equivalent %, see, e.g., \cite[Theorem XIII.5]{rs}, 
to the operator inequality
$(B-\mu_i I)(B-\mu_{i+1}I)\geq0$, which evidently holds. 
The equality here %, also giving the minimum of $\|\nabla \mu(x)\|\|x\|$, 
is attained only for $x\in\spanl\{x_i,x_{i+1}\}.$

Finally, we turn our attention to the angles. 
For $x\in{\mathcal  L}(\kappa),$ the Pythagorean theorem
$\|Bx\|^2=\|\kappa x\|^2+\|Bx-\kappa x\|^2$ 
shows that  %$a={\|Bx-\kappa x\|}/{\|Bx\|}\in[0,1)$ 
\[a^2:=\frac{\|Bx-\kappa x\|^2}{\|Bx\|^2}
=\frac{\|Bx-\kappa x\|^2/\|x\|^2}{\kappa^2+\|Bx-\kappa x\|^2/\|x\|^2}
\in(0,1)\]
is minimized together with $\|Bx-\kappa x\|/\|x\|$. 
But for a fixed $\gamma\in(0,1)$ the function 
$\arcsin(a)-\arcsin(\gamma a)$ is strictly increasing
in $a\in(0,1)$ which proves the proposition for
$\px-\pgx=\arcsin(a)-\arcsin(\gamma a).$
\end{proof}

Now we are ready to show that the same subspace $\spanl\{x_i,x_{i+1}\}$
gives the smallest change in the Rayleigh quotient $\mu(w)-\kappa$. 
The proof is based on analyzing the negative normalized gradient flow of the Rayleigh quotient.
\begin{theorem}
\label{t.2Dred}
Under the assumptions of Theorems \ref{t.wcurve} and \ref{t.2Dgrad} we denote 
$\igk:=\{w: w\in\arg\min\mu(\cgx);\; x\in {\mathcal  L}(\kappa)\}$---the 
set of minimizers of the  Rayleigh quotient.
Then $\arg\min\mu(\igk)\in\spanl\{x_i,x_{i+1}\}.$
(See Figure \ref{f2}). 
\begin{figure}%%%%%%%%%%%%%%%%%%%%%%%%%%%%%%%%%%%%%%%%%%%%%%%%%%%
\begin{center}
\includegraphics[totalheight=0.3\textheight]{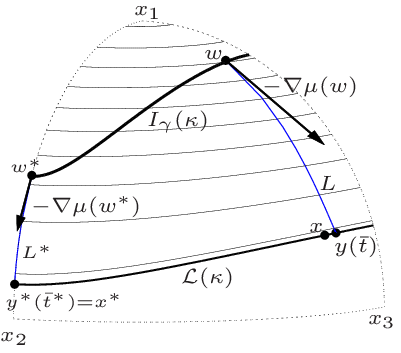}
\end{center}
\caption{\em The Rayleigh quotient gradient flow integration on the unit ball. 
}
\label{f2}
\end{figure}%%%%%%%%%%%%%%%%%%%%%%%%%%%%%%%%%%%%%%%%%%%%%%%%%%%%%%
\end{theorem}
\begin{proof}
The initial value problem for a gradient flow of the Rayleigh quotient,
\begin{equation}\label{e:gf}
y'(t)=-\frac{
\nabla\mu(y(t))}{\|\nabla\mu(y(t))\|}
,\, t\geq 0,\, y(0)=w\in\igk,
\end{equation}
has the vector-valued solution $y(t),$ which preserves 
the norm of the initial vector $w$ since $d\|y(t)\|^2/dt=2(y(t),y'(t))=0$
as $(y,\nabla\mu(y))=0$. 
Without loss of generality we assume $\|w\|=1=\|y(t)\|.$
The Rayleigh quotient 
function $\mu(y(t))$ is decreasing since
\[
     \frac{d}{dt} \mu(y(t))  = \left(\nabla \mu(y(t)), y'(t)\right)
    = \left(\nabla \mu(y(t)),
-         \frac{
\nabla\mu(y(t)) }{\|\nabla\mu(y(t))\|}
\right)
    =-\|\nabla \mu(y(t))\|\leq0.
\]
As $\mu(y(0))=\mu(w)<\mu_i$, the function $\mu(y(t))$ is strictly decreasing 
at least until it reaches $\kappa>\mu_{i+1}$ as there are no 
eigenvalues in the interval $[\kappa,\mu(y(0))]\subset(\mu_{i+1},\mu_i),$
but only eigenvectors can be special points of ODE \eqref{e:gf}. 
The condition $\mu(y(\bar t))=\kappa$ thus uniquely determines $\bar t$
for a given initial value $w$.
The absolute value of the decrease
of the Rayleigh quotient along the path $L:=\{y(t),\,0\leq t\leq\bar t\}$ is %given by
\[ 
\mu(w)-\kappa=\mu(y(0))-\mu(y(\bar t))
     =\int_0^{\bar t} \|\nabla \mu(y(t))\| dt >0. 
\] 
Our continuation method \eqref{e:gf} using the \emph{normalized} gradient flow is nonstandard, 
but its advantage is that it gives the following simple expression for the length of $L$, 
$ \mathrm{Length}(L)=\int_{0}^{\bar t} \|y'(t)\| dt=
   \int_{0}^{\bar t} 1 dt =\bar t.$

Since the initial value $w$ is determined by $x$, 
we compare a generic $x$ with the special choice $x=x^*\in\spanl\{x_i,x_{i+1}\}$,
using the superscript $*$ to denote all quantities corresponding to the 
choice  $x=x^*$. By Theorem~\ref{t.wcurve} %our choice of 
$x^*\in\spanl\{x_i,x_{i+1}\}$ implies $w^*\in\spanl\{x_i,x_{i+1}\}$, so
we have $y^*(t)\in\spanl\{x_i,x_{i+1}\},$ $0\leq t\leq \bar t^*$ as 
$\spanl\{x_i,x_{i+1}\}$ is an invariant subspace for 
the gradient of the Rayleigh quotient. 
At the end points,  
$\mu(y(\bar t))=\kappa=\mu(x)=\mu(y^*(\bar t^*)),$
by their definition. Our goal is to bound the initial value  $\mu(w^*)=\mu(y^*(0))$ by 
%the initial value 
$\mu(w)=\mu(y(0))$, %. To achieve it, 
so we compare the lengths of the corresponding paths $L^*$ and $L$
and the norms of the gradients along these paths. 

We start with the lengths. 
We obtain $\phi_1(x^*)-\phi_\gamma(x^*)\leq\px-\pgx$ by Theorem~\ref{t.2Dgrad}.
Here the angle $\px-\pgx$ is the smallest angle between any two vectors on the cones 
boundaries  $\partial\cgx$ and $\partial\cx$. 
Thus,  $\px-\pgx\leq\angle\{y(0),y(\bar t)\}$ 
as our one vector $y(0)=w\in\partial\cgx$ 
by Theorem~\ref{t.wcurve}, while the other vector $y(\bar t)$
cannot be inside the cone $\cgx$ 
since $\mu(w)>\kappa=\mu(y(\bar t))$ by Theorem~\ref{t.cone}. 
As $y(t)$ is a unit vector, 
$\angle\{y(0),y(\bar t)\}\leq\mathrm{Length}(L)=\bar t$ 
as the angle is the length of the arc---the shortest
curve from $y(0)$ to $y(\bar t)$ on the unit ball. 

For our  special $*$-choice, inequalities from the previous 
paragraph turn into equalities, 
as  $y^*(t)$ is in the intersection of the unit ball
and the subspace $\spanl\{x_i,x_{i+1}\}$, so the 
path $L^*$ is the arc between $y^*(0)$ to $y^*(\bar t^*)$ itself.  
Combining everything together, %we obtain 
\begin{align*}
\bar t^*=
\mathrm{Length}(L^*)=\angle\{y^*(0),y^*(\bar t^*)\}=
\angle\{w^*,x^*\}=\varphi_1(x^*)-\varphi_\gamma(x^*)\\
\leq
\varphi_1(x)-\varphi_\gamma(x)
\leq
\angle\{y(0),y(\bar t)\}
\leq
\mathrm{Length}(L)=\bar t.
\end{align*}

By Theorem \ref{t.2Dgrad} on 
the  norms of the gradient, 
$-\|\nabla \mu(y^*(t^*))\|\geq-\|\nabla \mu(y(t))\|$
 for each pair of independent variables $t^*$ and $t$ such that
$\mu(y^*(t^*))=\mu(y(t)).$
Using Theorem \ref{t:iif}, we conclude that 
$\mu(w^*)=\mu(y^*(0))\leq\mu(y(\bar t-\bar t^*))\leq \mu(y(0))=\mu(w)$ 
as $\bar t-\bar t^*\geq0$,
%\[ \mu(w^*)=\mu(y^*(0))\leq 
%   \mu(y(\underbrace{\bar t-\bar t^*}_{\geq 0}))
%   \leq \mu(y(0))=\mu(w),\]
i.e.,\ the subspace $\spanl\{x_i,x_{i+1}\}$ gives the smallest value $\mu(w).$ 
\end{proof}

By Theorem~\ref{t.2Dred} the poorest convergence is attained with 
$x\in\spanl\{x_i,x_{i+1}\}$ and with the corresponding minimizer $w\in\spanl\{x_i,x_{i+1}\}$ 
described in Theorem \ref{t.wcurve}, so finally our analysis is now reduced 
to the two-dimensional space $\spanl\{x_i,x_{i+1}\}$.
\begin{theorem}\label{t.2D}
Bound \eqref{e.muestn} holds and is sharp for $x\in\spanl\{x_i,x_{i+1}\}$.
\end{theorem}
\begin{proof}
Assuming $\|x\|=1$ and $\|x_i\|=\|x_{i+1}\|=1$, we derive 
\begin{equation}
   \label{e.xixip}
 |(x,x_i)|^2=\frac{\mu(x)-\mu_{i+1}}{\mu_i-\mu_{i+1}}>0
  \mbox{ and }
  |(x,x_{i+1})|^2=\frac{\mu_i-\mu(x)}{\mu_i-\mu_{i+1}},
\end{equation}
and similarly for $w\in\spanl\{x_i,x_{i+1}\}$ where $(B+\alpha I)w=Bx$.

Since $B>0$, we have $x=(I+\alpha B^{-1})w$. 
Assuming $\alpha=-\mu_{i+1}$, this identity implies 
$x=x_i,$ which contradicts our assumption that 
$x$ is not an eigenvector. 
For $\alpha\neq-\mu_{i+1}$ and $\alpha>-\mu_i$ by Theorem~\ref{t.wcurve},
the inverse $(B+\alpha I)^{-1}$ exists. 

Next we prove that $\alpha>0$ and that it  
is a strictly decreasing function of $\kappa:=\mu(x)\in(\mu_{i+1},\mu_i).$
Indeed, using $Bx=(B+\alpha I)w$ and our cosine-based definition of the angles, we have 
$0<(w,(B+\alpha I)w)^2=(w,Bx)^2=\|w\|^2\|Bx\|^2\cos^2\pgx,$ where %, using the sine,
$\|Bx\|^2\cos^2\pgx=\|Bx\|^2-\gamma^2 \|Bx-\kappa x\|^2.$
We substitute $w=(B+\alpha I)^{-1}Bx$, which gives  
$((B+\alpha I)^{-1}Bx,Bx)^2=\|(B+\alpha I)^{-1}Bx\|^2
\left(\|Bx\|^2-\gamma^2 \|Bx-\kappa x\|^2\right).$
Using \eqref{e.xixip}, multiplication by $(\mu_i+\alpha)^2(\mu_{i+1}+\alpha)^2$
leads to a simple quadratic equation,
%\[\gamma^2(\kappa(\mu_i+\mu_{i+1})-\mu_i\mu_{i+1})\alpha^2
%    +2\gamma^2\kappa\mu_i\mu_{i+1}\alpha
%    -(1-\gamma^2)\mu_i^2\mu_{i+1}^2=0,\] 
$a\alpha^2+b\alpha+c=0,\,a=\gamma^2(\kappa(\mu_i+\mu_{i+1})-\mu_i\mu_{i+1}),\,
b=2\gamma^2\kappa\mu_i\mu_{i+1},\,c= -(1-\gamma^2)\mu_i^2\mu_{i+1}^2$
for $\alpha.$ As $a>0$, $b>0$, and $c<0$, the discriminant is
positive and the two solutions for $\alpha$, 
corresponding to the minimum and maximum of the Rayleigh quotient on $\cgx$, 
have different signs. The proof of Theorem \ref{t.wcurve} analyzes 
the direction of the gradient of the Rayleigh quotient to conclude that $\beta>0$ and 
$\alpha>-\mu(w)$ correspond to the minimum. Repeating the same arguments 
with $\beta<0$ shows that  
$\alpha<-\mu(w)$ corresponds to the maximum. But $\mu(w)>0$ since $B>0$, hence 
the negative $\alpha$ corresponds to the maximum and thus 
the positive $\alpha$ corresponds to the minimum. 
We observe that the coefficients $a>0$ and $b>0$ 
are evidently increasing functions of $\kappa\in(\mu_{i+1},\mu_{i})$, 
while $c<0$ does not depend on $\kappa$. 
Thus $\alpha>0$ is strictly decreasing in $\kappa$, 
and taking $\kappa\to\mu_i$ gives the smallest  
$\alpha=\mu_{i+1}(1-\gamma)/\gamma>0.$ 
%\[ 
%\alpha=\frac{\mu_i\mu_{i+1}\kappa}{\kappa(\mu_i+\mu_{i+1})-\mu_i\mu_{i+1}}
%    \left( \sqrt{
%    1+\frac{1-\gamma^2}{\gamma^2}
%\left(1+\frac{(\mu_i-\kappa)(\kappa-\mu_{i+1})}{\kappa^2}\right)}-1\right)>0
%\]

Since $(B+\alpha I)w=Bx$ where now $\alpha>0$, condition 
$(x,x_{i})\neq0$ implies $(w,x_{i})\neq0$ and 
$(x,x_{i+1})=0$ implies $(w,x_{i+1})=0,$
so we introduce the convergence factor as 
\[
\sigma^2(\alpha):=
\frac{\mu_i-\mu(w)}{\mu(w)-\mu_{i+1}}
\frac{\mu(x)-\mu_{i+1}}{\mu_i-\mu(x)}
=
\left|\frac{(w,x_{i+1})}{(w,x_{i})}\right|^2
\left|\frac{(x,x_{i})}{(x,x_{i+1})}\right|^2
=
\left(\frac{\mu_{i+1}}{\mu_i}\frac{\mu_i+\alpha}{\mu_{i+1}+\alpha}\right)^2, 
\]
where we use \eqref{e.xixip} and  again $(B+\alpha I)w=Bx$. 
We notice that $\sigma(\alpha)$ is a strictly decreasing function of $\alpha>0$ 
and thus takes its largest value for $\alpha=\mu_{i+1}(1-\gamma)/\gamma$
giving $\sigma=\gamma+(1-\gamma){\mu_{i+1}}/{\mu_i},$ i.e.,\ 
bound \eqref{e.muestn} that we are seeking.

The convergence factor $\sigma^2(\alpha)$ cannot be improved 
without introducing extra terms or assumptions. 
But $\sigma^2(\alpha)$ deals with $w\in\cgx$, not with the actual  
iterate $x'$. We now show that for %a given matrix $B>0$  
%(which is real diagonal in the basis of its eigenvectors) and 
$\kappa\in(\mu_{i+1},\mu_i)$ there exist a vector 
$x\in\spanl\{x_i,x_{i+1}\}$ and 
a preconditioner $T$ satisfying \eqref{e.ass} such that
$\kappa=\mu(x)$ and $x'\in\spanl\{w\}$ in both real and complex cases. 
%Choosing $x=|x|$, we show that in the complex case the complex vector $w\in\cgx$  
%has a real substitute, namely $|w|\in\cgx$, having the same Rayleigh quotient. 
%Then we prove that each {\em real} vector in the (real or complex) cone $\cgx$ 
%is generated by a proper preconditioner $T$.
In the complex case, let us choose $x$ such that  $\mu(x)=\kappa$ and $x=|x|$ according to 
\eqref{e.xixip}, %---such a choice is always possible. %and is, in fact, unique. 
%In the real case, $w=(B+\alpha I)^{-1}Bx$ with $\alpha>0$, so 
%$w=|w|$ automatically. 
then the real vector $w=|w|\in\cgx$ is a minimizer 
of the Rayleigh quotient on $\cgx$, since 
 $\mu(w)=\mu(|w|)$ and $|(w,B|x|)|\leq(|w|,B|x|)$. 

Finally, for a real $x$ with  $\mu(x)=\kappa$ and 
a real properly scaled $y\in\cgx$ 
there is a real matrix $T$ satisfying \eqref{e.ass} such that
$y=Bx-(I-T)(Bx-\kappa x)$, which leads to \eqref{e.ep} with $\mu(x)x'=y.$ 
Indeed, for the chosen $x$ we scale $y\in\cgx$ such that 
$(y,Bx-y)=0$ so $\|Bx-y\|=\sin\pgx\|Bx\|=\gamma\|Bx-\kappa x\|$. As vectors 
$Bx-y$ and $\gamma(Bx-\kappa x)$ are real and have the same length there 
exists a \emph{real} Householder reflection $H$ such that $Bx-y=H\gamma(Bx-\kappa x)$. 
Setting $T=I-\gamma H$ we obtain the required identity. Any Householder 
reflection is symmetric and has only two distinct eigenvalues $\pm1$, so 
we conclude that $T$ is real symmetric (and thus Hermitian in the complex case) 
and satisfies \eqref{e.ass}.
\end{proof}
%%%%%%%%%%%%%%%%%%%%%%%%%%%%%%
\section{Appendix}
\label{secA}
The integration of inverse functions theorem follows.
\begin{theorem}\label{t:iif}
Let $f,\: g:[0,b]\to{\bf R}$ for $b>0$ be strictly monotone increasing smooth
functions and suppose that for $a\in[0,b]$ we have $f(a)=g(b)$. If for all
$\alpha,\:\beta\in[0,b]$ with $f(\alpha)=g(\beta)$ the derivatives satisfy
$f'(\alpha) \leq g'(\beta),$
then for any $\xi\in[0,a]$ we have $f(a-\xi) \geq g(b-\xi).$
\end{theorem}
\begin{proof}
For any $\xi\in[0,a]$ we have (using $f(a)=g(b)$)
\[
  \xi=\int_{g(b-\xi)}^{g(b)} \left(g^{-1}\right)'\!(y) \:dy =
       \int_{f(a-\xi)}^{g(b)} \left(f^{-1}\right)'\!(y) \:dy.
\]
If $y=f(\alpha)=g(\beta)$, then for the derivatives of the inverse functions
it holds that
$
  \left(g^{-1}\right)'\!(y) \leq \left(f^{-1}\right)'\!(y).
$
Since $f$ and $g$ are strictly monotone increasing functions the integrands
are positive functions and  $g(b-\xi)<g(b)$ as well as
$f(a-\xi)<f(a)=g(b)$. Comparing the lower
limits of the integrals gives the statement of the theorem. 
\end{proof}
%%%%%%%%%%%%%%%%%%%%%%%%%%%%%%
\section*{Conclusions}
\label{secC}
We present a new geometric
approach to the convergence analysis of
a preconditioned fixed-step gradient eigensolver 
which reduces the derivation of the convergence rate bound 
to a two-dimensional case. 
The main novelty is in the use of a continuation
method for the gradient flow of the Rayleigh quotient 
to locate the two-dimensional subspace corresponding to the 
smallest change in the Rayleigh quotient and 
thus to the slowest convergence of the gradient eigensolver. 

An elegant and important result such as Theorem \ref{t.1} should ideally have 
a textbook-level proof. We have been trying, unsuccessfully, to find such a proof  
for several years, so its existence remains an open problem.  
%%%%%%%%%%%%%%%%%%%%%%%%%%%%%%
\section*{Acknowledgments}
%We thank the anonymous referees in advance for their 
%useful and constructive suggestions.
We thank M. Zhou of University of Rostock, Germany for proofreading. 
M. Argentati of University of Colorado Denver, 
E. Ovtchinnikov of University of Westminster, 
and anonymous referees have made numerous great suggestions 
to improve the paper and for future work. 

%the paper and numerous comments.  

\def\refname{\centerline{\footnotesize\rm REFERENCES}}
%\bibliographystyle{siam}
%\bibliographystyle{plainnat}
%\bibliography{lit}

\end{document}